\documentclass[12pt]{article}

\usepackage{amsmath, amsfonts, amssymb}

\usepackage[english]{babel}
\usepackage[utf8]{inputenc}
\usepackage{todonotes}
\usepackage{comment}
\usepackage[left=2.5cm, right=3cm, top=2cm, bottom=3cm, bindingoffset=0cm, marginparwidth=3cm]{geometry}

\usepackage{hyperref}
\usepackage{enumitem}
\usepackage{amsthm}

\usepackage{tikz}
\usetikzlibrary{arrows,calc}
\tikzset{
>=stealth',
}

\newcounter{theorems}
\newtheorem{prop}[theorems]{Proposition}
\newtheorem{thm}[theorems]{Theorem}

\newtheorem{cor}[theorems]{Corollary}

\newtheorem{lem}[theorems]{Lemma}

\newtheorem*{theorem*}{Theorem}

\theoremstyle{definition}

\theoremstyle{remark}

\newtheorem{rem}[theorems]{Remark}
\newtheorem*{rem*}{Remark}

\makeatletter
\def\blfootnote{\gdef\@thefnmark{}\@footnotetext}
\makeatother

\def\e {\varepsilon }

\def\phi {\varphi}

\def\be {\begin{equation}}

\def\ee {\end{equation}}

\def\bt {\begin{thm}}

\def\et {\end{thm}}

\def\R {\mathcal R}

\def\nn {\mathbb N}

\DeclareMathOperator{\dist}{dist}

\def\U{\mathcal U}

\def\t{{\theta}}

\begin{document}

\title{Generic families of circle diffeomorphisms have many coexisting periodic orbits}
\author{Ivan Shilin\footnote{HSE University, Moscow}}
\date{}

\maketitle
\begin{abstract}
We prove that for a generic family of circle diffeomorphisms every parameter value that corresponds to an irrational rotation number is approximated by parameter values for which the diffeomorphisms have arbitrarily large finite numbers of periodic orbits. This phenomenon implies that families where irrational rotation numbers appear are not weakly structurally stable. Moreover, we prove that any locally residual set of one-parameter families with nonconstant rotation number yields a continuum of weak equivalence classes of families.
\end{abstract}



\blfootnote{This study was supported by the Basic Research Program of the HSE University (project No. 075-00648-25-00 ``Symmetry. Information. Chaos'').}

We consider finite-parametric families of orientation preserving circle diffeomorphisms, and our primary focus is on one-parameter families. In the general case of $l$-parameter families, we use notation $F = (f_\beta)_{\beta\in B}$, where $B$ is the parameter space that is usually assumed to be a closed ball or a rectangle in~$\mathbb{R}^l$. In the one-parameter case we write $F = (f_\theta)_{\theta\in I},\; I = [-1, 1].$ The exact regularity in the parameter is not important for most of the arguments, but one may assume that we are dealing with $C^k$-families of $C^r$-smooth circle diffeomorphisms, $1 \le k \le r \le +\infty$. The rotation number of a circle diffeomorphism $f_\beta$ is denoted~$\rho(f_\beta)$; it depends continuously on the diffeomorphism and hence on the parameter, see~\cite{KH} for the modern exposition of the relevant classical results. The properties of the rotation number and related properties of circle diffeomorphisms were studied extensively, e.g., in~\cite{A, He, Y, KhT, B2}. Bifurcations that happen in generic one-parameter families of diffeomorphisms were studied in~\cite{S, NP, B1, B2, NPT}, to name a few. We focus on a somewhat neglected question of how many periodic orbits can coexist, in a generic family of circle diffeomorphisms, for a parameter value that corresponds to a rational rotation number.
%The main result, to put it simply, is that a generic family of orientation preserving circle diffeomorphisms for which the rotation number is not constant has diffeomorphisms with an arbitrary number of hyperbolic sources.

\begin{thm}\label{thm:main}
For a generic $l$-parameter $C^k$-family of $C^r$-smooth orientation preserving circle diffeomorphisms, $r\ge 1$, and for every integer $N>0$, every parameter value that corresponds to an irrational rotation number is accumulated by parameter values that correspond to structurally stable diffeomorphisms with exactly~$2N$ periodic orbits.
\end{thm}

The case $r \ge 2$ is more convenient to deal with, because this regularity is needed to define quadratic saddle-node bifurcations of periodic orbits. For $r \ge 2$, the results of P.~Brunovsk{\'y}~\cite{B1, B2} imply that generic one-parameter families of circle diffeomorphisms contain only Kupka-Smale diffeomorphisms and diffeomorphisms with quadratic saddle-node bifurcations; the latter appear for isolated values of the parameter and have only one parabolic orbit.\footnote{This also follows from the results of J.~Sotomayor~\cite{S}, as explained in~\cite{NP}.} In particular, for a generic one-parameter family each rational rotation number appears on a union of disjoint segments in the parameter space, and at the endpoints of each segment a single parabolic orbit is present. Since each saddle-node bifurcation produces exactly one pair of hyperbolic periodic orbits, Theorem~\ref{thm:main} can be strengthened for $r \ge 2$: for a generic family, each parameter value that corresponds to an irrational rotation number is also accumulated by parameter values that correspond to any given positive \emph{odd} number of periodic orbits, all hyperbolic except one. For one-parameter families, it follows from Theorem~\ref{thm:main} and the results mentioned above, and the argument for $l$-parameter families will be given in Proposition~\ref{prop:all_numbers} below.

Two families $F = (f_\beta)_{\beta \in B}$, $G = (g_{\tilde{\beta}})_{\tilde{\beta} \in \tilde{B}}$ of diffeomorphisms of a phase space $M$ are called \emph{weakly topologically equivalent} if there exists a map
\begin{equation}\label{eq:weak_eqiv}
H : B \times M \to \tilde{B} \times M, \quad H(\beta, x) = \bigl(h(\beta), H_\beta(x)\bigr)
\end{equation}
such that $h : B \to \tilde{B}$ is a homeomorphism and for each $\beta \in B$ the map 
$H_\beta \colon M \to M$ is a homeomorphism that conjugates $f_\beta$ and $g_{h(\beta)}$. We talk about strong equivalence if $H$ itself is a homeomorphism.

It is not obvious how many classes of weak equivalence a residual or locally residual set of one-parameter families can provide. A priori one must keep in mind that the answer may depend on a chosen residual set. Of course, the cardinality of the set of equivalence classes cannot exceed the cardinality of the whole space of families, i.e., the cardinality of continuum. 
For brevity, let us call a family \emph{nontrivial} if the rotation number is not constant as a function of the parameter, i.e., it nontrivially depends on~$\beta$.
We will prove that for $r \ge 2$ any locally residual set of nontrivial one-parameter families provides a set of equivalence classes of exactly cardinality continuum, and we conjecture that the same holds for $l$-parameter families as well.

\begin{thm}\label{thm:continuum}
For any locally residual subset of nontrivial $C^k$-smooth one-parameter families of $C^r$-smooth orientation preserving circle diffeomorphisms, $2 \le r \le +\infty$, the set of weak topological equivalence classes has cardinality continuum.
\end{thm}

%In the case of $r=1$ the same applies for any open set of families that contains a nontrivial family, see Remark~\ref{rem:r1}.

All results of the present paper can be reformulated for the case of finite-parameter families of nonsingular flows on the 2-torus with a global cross section.

\section{Arbitrarily large numbers of coexisting sourses}

\subsection{The proof of Theorem~\ref{thm:main}}
We assume that the space of families that we consider is endowed with a metric.
The proof of Theorem~\ref{thm:main} is based on the following perturbation lemma.
\begin{lem}\label{lem:pert}
    Let $(g_\beta)$ be a family such that at $\beta_0$ the rotation number $\rho(g_{\beta_0})$ is irrational. Then for any $N \in \nn$ and $\e > 0$ the family $(g_\beta)$ can be approximated by a family $(f_\beta)$ that has a structurally stable diffeomorphism with exactly $2N$ hyperbolic periodic orbits at some~$\beta_1$ that is $\e$-close to~$\beta_0$.
\end{lem}
\begin{proof}
    If the rotation number is constant in a neighborhood of $\beta_0$, the family can be perturbed so that it becomes nonconstant. Furthermore, we can make the perturbed family $C^\infty$-smooth in both the parameter and the phase variable. After this perturbation, there is $\beta_1$ in an $\e$-neighborhood of $\beta_0$ such that $\rho(f_{\beta_0})$ is Diophantine. By the Herman--Yoccoz theorem (we refer to the Yoccoz version from~\cite{Y}),
the new map $g_{\beta_1}$ is $C^\infty$-smoothly conjugate with a pure rotation. In these normalizing smooth coordinates,
the map $g_{\beta_1}$ can be turned into a rational rotation $R_{p/q}$ by an arbitrarily small perturbation. Another arbitrarily small perturbation turns this rational rotation into a structurally stable map with $2N$ periodic orbits: it suffices to take a composition of $R_{p/q}$ with a time-$\delta$ map of a $R_{p/q}$-invariant flow that has $N$ hyperbolic sinks and $N$ hyperbolic sources in each fundamental domain of~$R_{p/q}$. Such two-step perturbation of $g_{\beta_1}$ can be realized by an arbitrary small perturbation of the family. The perturbed family $(f_\beta)$ has exactly~$2N$ hyperbolic periodic orbits for~$f_{\beta_1}$.
\end{proof}

\begin{rem}\label{rem:supp}
We can ensure that the perturbation in Lemma~\ref{lem:pert} is supported in an arbitrarily small neighbourhood of $\beta_0$ in the parameter space, provided that this neighborhood is chosen in advance. The only modification in the proof is that the perturbed family becomes $C^\infty$-smooth only in a neighborhood of $\beta_0$, which is enough to apply the Yoccoz theorem without loosing regularity.
\end{rem}

\begin{proof}[Proof of Theorem~\ref{thm:main}]

Let $(U_n)_{n\in\nn}$ be a countable base of topology in the parameter space~$B$. Fix some $U_n$ and consider an arbitrary family $F = (f_\beta)_{\beta\in B}$. For each $U_n$, either $\rho(f_\beta)$ is rational for every $\beta \in U_n$, or there is $\beta_0 \in U_n$ for which $\rho(f_{\beta_0})$ is irrational.
In the second case, Lemma~\ref{lem:pert} applied to a one-parameter subfamily passing through $f_{\beta_0}$ yields a perturbed family that has exactly $2N$ orbits at some parameter value $\beta_1 \in U_n$, all hyperbolic.\footnote{Note, however, that at another parameter value that corresponds to the same rational rotation number our perturbed family may have more than~$2N$ hyperbolic periodic orbits.} 

Denote by $\mathcal{V}_{n,N}$ the open set of families that, for some $\beta \in U_n$, have a structurally stable diffeomorphism with exactly~$2N$ periodic orbits and denote by $C_n$ the set of families that have only rational rotation numbers for $\beta\in U_n$.\footnote{If $U_n$ is connected, $\rho(f_\beta)$ is constant on~$U_n$ for $F \in C_n$, but we do not use this in the proof.} Observe that any family that is not in $\operatorname{Int}C_n$ can be approximated by a family not in $C_n$ and another perturbation yields a family in $\mathcal{V}_{n,N}$, via Lemma~\ref{lem:pert}. Hence, the set $\mathcal{U}_{n,N} = (\operatorname{Int}C_n) \cup \mathcal{V}_{n,N}$ is open and dense in the space of families. For a family $F = (f_\beta)$ in the residual set
\[
\mathcal{R} = \bigcap_{n,N} \mathcal{U}_{n,N} =
\bigcap_n \left(\operatorname{Int} C_n \cup \bigcap_N \mathcal{V}_{n,N}\right)
\]
we have the following: for each~$U_n$, either $\rho(f_\beta)$ is rational on the whole~$U_n$, or for each $N$ there is $\beta_N \in U_n$ such that $f_{\beta_N}$ is structurally stable and has exactly~$2N$ periodic orbits.

Consider a family $F \in \mathcal{R}$ and a $\beta_0$ such that $\rho(f_{\beta_0})$ is irrational. For any neighborhood $W$ of $\beta_0$ there is $U_n$ such that $\beta_0 \in U_n \subset W$. Hence, for each $N\ge 1$, this $U_n$ contains a $\beta_N$ that corresponds to a structurally stable diffeomorphism with exactly~$2N$ hyperbolic periodic orbits. That is, $\beta_0$ is approximated by parameter values at which $f_\beta$ has~$2N$ hyperbolic periodic orbits.
\end{proof}

\subsection{Any number of periodic points is observed when \texorpdfstring{$r \ge 2$}{r >= 2}}

\begin{prop}\label{prop:all_numbers}
For a generic $l$-parameter $C^r$-family of orientation preserving circle diffeomorphisms, with $2 \le r \le +\infty$, for every integer $N \in \mathbb N \cup \{0\}$, every parameter value that corresponds to an irrational rotation number is accumulated by parameter values that correspond to diffeomorphisms with exactly~$N$ periodic orbits.
\end{prop}
\begin{proof}
For even $N>0$ the claim follows from Theorem~\ref{thm:main}. The claim for $N=0$ is trivial since every parameter value that corresponds to an irrational rotation number is accumulated by other parameter values that correspond to irrational rotation numbers and hence to $0$ periodic orbits. For the case of 1-parameter families the claim for odd $N$ follows from Theorem~\ref{thm:main} and the characterization of generic families in~\cite{B1}.\footnote{The result of~\cite{B1} does not cover the case $r = +\infty$, but this gap is covered in the proof of Lemma~\ref{lem:trans} below.} We reduce the general case of odd $N$ to the case of 1-parameter families.

Consider a segment $\gamma$ in the parameter base $B$ and identify it with $I = [-1,1]$ by an affine map. Fix a residual set $\mathcal{R}$ of 1-parameter $C^r$-families. It is easy to see that the set $\tilde{\mathcal{R}}$ of $l$-parameter families that yield an element of $\mathcal{R}$ in restriction to $\gamma$ is residual. Indeed, without loss of generality $\mathcal{R}$ decomposes as $\cap_j \mathcal {U}_j$, where $\mathcal {U}_j$ are open and dense. Restriction to $\gamma$ is a continuous map from $l$-parameter families to $1$-parameter families, so the preimages of $\mathcal {U}_j$ are open. The preimages are dense because a small perturbation of a restriction to~$\gamma$ can be realized by a small perturbation of the $l$-parameter family.

Now, consider a countable dense set in $B$ and join every pair of points in it by a segment. Enumerate these segments to obtain a set of segments $(\gamma_n)_{n\in{\mathbb N}}$ and apply the previous construction with the same $\mathcal R$. Namely, $\mathcal R$ consists of families whose non-Morse-Smale diffeomorphisms with rational rotation number are those with a single nondegenerate parabolic orbit, with a saddle node bifurcation unfolding when the parameter changes. We obtain residual sets $\tilde{\mathcal{R}}_n$ of $l$-parameter families that have the same property in restriction to~$\gamma_n$. The intersection $\tilde{\mathcal{R}} = \cap_n \tilde{\mathcal{R}}_n$ is the required residual set of $l$-parameter families. Indeed, let $F \in \tilde{\mathcal{R}}$ and let $\rho(f_{\beta_0})$ be irrational. Then, as was shown in the proof of Theorem~\ref{thm:main}, in any neighborhood $W$ of $\beta_0$ we can find values $\beta_1, \beta_2$ that correspond to structurally stable diffeomorphisms with different rotation numbers. Moreover, we can take $\beta_1, \beta_2$ being the endpoints of some $\gamma_n$ and, moreover, with $f_{\beta_1}$ having~$2N$ hyperbolic orbits. Then, in restriction to $\gamma_n$ we obtain a 1-parameter family with nonconstant rotation number, where the number of periodic orbits must drop from $2N$ to $1$ at saddle-node bifurcation points before the rotation number can start to change. Hence, all possible numbers of periodic orbits from $2N$ to $1$, in particular~$N$, are observed for $f_\beta$ with ${\beta \in W}$.
\end{proof}

\subsection{Corollaries on structurally unstable families}
It is not difficult to see that among finite-parameter families some are weakly structurally stable: e.g., it is possible that for $F = (f_\beta)$ all $f_\beta$ are structurally stable. A less trivial example can be seen for one-parameter families for $r\ge 2$: consider a family where the rotation number is constant and rational and the maps $f_\t$ are structurally stable for all $\t$ except a finite set of values where saddle-node bifurcations happen. However, global families with nonconstant rotation number form an open set, and we will show that families in this open set are not weakly structurally stable.

\begin{prop}\label{prop:unstable}
If the rotation number assumes some irrational values in a family $F = (f_\beta)_{\beta \in B}$, then this family is not weakly structurally stable.
\end{prop}
\begin{proof}
Consider $\beta_0\in I$ such that $\rho(f_{\beta_0})$ is irrational. Arguing as in the proof of Lemma~\ref{lem:pert}, we obtain a perturbed family $\tilde{F}$ such that there is $\beta_1$ near $\beta_0$ with $\tilde{f}_{\beta_1}$ being a pure \emph{rational} rotation in some smooth coordinate. On the other hand, there is a family $G$ near $F$ that has at most finite number of periodic orbits at each value of the parameter; such families are generic\footnote{For one-parameter families and finite~$r$ this genericity is established in~\cite{B1}. It is not difficult to prove it in general, see Remark~\ref{rem:isolated} below.} in the space of $C^r$-families if $r > l+1$ and hence dense in the space of $C^k$-families of $C^r$ circle diffeomorphisms for all~$k, r$. Therefore, $F$ is not structurally stable.
\end{proof}

\begin{rem}
The strong structural stability of families in which the rotation number is not constant was essentially ruled out by~\cite[Ch.III.5]{NPT}. %Since~\cite{NPT} does not have exactly this claim, some explanation is in order. Under the assumption that the family is strongly stable, the endpoints $a, b$ of $I$ must correspond to structurally stable diffeomorphisms, and the other endpoint of the maximal segment $[a, c] \subset I$ where the rotation number is constant corresponds to a map with only one periodic orbit, which is parabolic. On the one hand, in the absence of other periodic orbits this parabolic orbit generates a Mather invariant of smooth conjugacy, and on the other hand, a theorem of Newhouse, Palis, and Takens says that the homeomorphism that establishes the strong equivalence of the families must also establish $C^1$-conjugacy of the maps with parabolic orbits. If the Mather invariants of $f_c$ and the $\tilde{f}_{\tilde{c}}$ do not coincide, the families will not be strongly equivalent, but the Mather invariant is changed by a small perturbation, which yields a contradiction. 
Of course, now this claim can be viewed as a trivial corollary of Proposition~\ref{prop:unstable}.
\end{rem}

It is not difficult to see that a generic germ of our family, say, at $\beta = 0$ has a structurally stable diffeomorphism at $\beta = 0$ and hence is itself strongly structurally stable. However, if we fix an irrational $\xi$ and consider the space of germs such that $\rho(f_0) = \xi$, we will see that generic germs there are not weakly structurally stable.

\begin{prop}
Generic local families $(f_\beta)_{\beta\in(\mathbb{R}^l, 0)}$ with $\rho(f_0) = \xi \notin \mathbb{Q}$ are not weakly structurally stable.
\end{prop}
\begin{proof}
Consider a representative $F$ of any such germ and a one-parameter subfamily $(f_\theta)$ passing through~$0$. Fix a sequence $\t_k \to 0$ of parameter values with irrational $\rho(f_{\t_k})$. By Remark~\ref{rem:supp}, we can perturb the family in disjoint neighborhoods~$I_k$ of the parameter values $\t_k$ so that the new one-parameter family $(\widehat{f}_\t)$ would have parameter values $\widehat{\t}_k \in I_k$ such that $\widehat{f}_{\widehat{\t}_k}$ is conjugate to a pure rational rotation. This perturbation can be realized as a perturbation of the original $l$-parameter $F$. However, the initial germ can also be approximated by a germ that does not admit a sequence $\beta_k \to 0$ that corresponds to diffeomorphisms with infinitely many periodic orbits. Hence, the initial germ cannot be weakly structurally stable.
\end{proof}

\section{A continuum of equivalence classes for generic \mbox{families}}

In this section, when we write ``families'', we mean one-parameter families of orientation preserving diffeomorphisms of~$S^1$. Unless explicitly stated otherwise, we assume that the regularity of the diffeomorphisms is at least~$C^2$, that is, we consider the space of $C^k$-smooth families of $C^r$-diffeomorphisms with~$r \ge 2$. 

\medskip

As was mentioned in the introduction, for any set of families, the cardinality of the set of equivalence classes cannot exceed the cardinality of continuum. On the other hand, it is not difficult to see that for any residual set of families with nonconstant rotation number the set of equivalence classes is uncountable.\footnote{Informally, one can assume the contrary, enumerate the rational numbers, $\mathbb Q\cap [0, 1) = \{q_n\}$, associate with each family a sequence $(a_n)_{n \in \mathbb N}$, where $a_n$ is the maximal number of periodic orbits that appear for diffeomorphisms of rational rotation numbers in a set $\{q_1, q_2, \dots, q_n\}$ inside the family, and then argue that Lemma~\ref{lem:pert} allows to show that a generic family has sequence $(a_n)$ that grows faster than any sequence in an arbitrary given countable family of monotonic sequences, which yields a contradiction.}
But Theorem~\ref{thm:continuum} is a bit more difficult to prove, so we start with a sketch of a proof.

\subsection{Idea of the proof}
We consider an arbitrary nonempty open set $\mathcal U$ of families and a residual subset $\mathcal R$ inside it. Without loss of generality, we may replace $\mathcal U$ by a proper nonempty open subset, and replace $\mathcal R$ by a proper subset residual in $\mathcal U$. Hence, after replacing $\mathcal U$ by a small neighborhood of a family  with nonconstant rotation number, we can assume that there is a segment $J \subset [0,1)$ such that any rotation number in $J$ appears in any family $F \in \mathcal U$. Moreover, after replacing $\R$ by a proper subset, we can assume that $\R = \cap_{n\ge 1} \U_n$, where $\mathcal{U}_n \subset \mathcal{U}$ form a nested sequence of open subsets: $\mathcal{U}_{n+1} \subset \mathcal U_n$.

We will construct a mapping $\Xi$ from the set of finite words of 0s and 1s into $\R$. A sequence $\omega$ from $\{0, 1\}^{\nn}$ can be viewed as a nested sequence of finite words. Each such sequence of words will be mapped into a Cauchy sequence with limit $F_\omega$ in~$\R$. These limit families $F_\omega$ will belong to pairwise different classes of weak equivalence.

In order to describe the invariant that will distinguish the non-equivalent families, we first fix an enumeration of rational numbers in $J$: $J \cap \mathbb Q = \{q_n\mid n\in\nn\}$ and, given a family $F \in \U$, denote by $a_n$ the maximal number of sources that coexist in $F$ for a parameter value that corresponds to the rational rotation number~$q_n$. In principle, it is possible that $a_n = +\infty$, but we will show in Lemma~\ref{lem:parity_stable} below that for a generic family these numbers are finite. Hence, replacing $\R$ by its subset (still residual in $\U$) and taking $F \in \R$, we make sure that we obtain a sequence of numbers. Also, we will deal only with families that have only hyperbolic sources, but in the definition of $a_n$ we count non-hyperbolic sources as well. Denote by $b_n$ the parity of $a_n$: $b_n = [a_n]_2 \in \{0,1\}$.
The parities sequence $(b_n)$ is an obvious invariant of weak topological classification of families.

For a limit family $F_\omega$ the sequence $b = (b_n)$ in general \emph{does not} coincide with~$\omega$. Informally, when we map a nested sequence of finite binary word into families, we set up some elements of the parities sequence using Lemma~\ref{lem:pert}, but we have poor control over which element we are allowed to change next. Moreover, if we want some elements of the sequence to stay fixed, this creates a constraint on further perturbation sizes and also affects which parity sequence element can be modified next.

Here is a game model of this, that will serve us only as an analogy. Two players play a game where they, together, construct a binary sequence in $\{0, 1\}^{\nn}$. The first player starts by writing an arbitrary nonempty binary word, the second responds with a single-digit word, then it is again the 1st player's turn, and so on; after this countable procedure the \emph{final binary sequence} appears. The question is if it is possible for the first player to announce a set $A \subset \{0, 1\}^{\nn}$ with cardinality less than continuum and then ensure that the final sequence sits in~$A$.

The answer is negative. Consider \emph{the strategy of $A$}: a mapping that takes a binary word $w$ into a move (also a binary word) that the 1st player would make given the current finite sequence written by the two players is~$w$; only deterministic strategies are allowed. If one knows this strategy and the final sequence, one can decode the moves made by the second player. Hence, if the strategy of the 1st player is fixed, the map that takes the sequence of the 2nd player's moves into the final sequence is injective, so the set of possible final sequences is of cardinality continuum. 

In the formal argument below, we play the role of the 2nd player, and the 1st player is ``nature'', or rather the axion of choice, that provides an element of the residual set $\R$ that satisfies given constraints on the parity sequence and perturbation size.

\subsection{The two lemmas}
The proof of Theorem~\ref{thm:continuum} relies on the following two lemmas, the proofs of which will be given in Sections~\ref{sec:p_lemma} and~\ref{sec:s_lemma} respectively.

\begin{lem}\label{lem:parity_stable}
There is a residual set $\widehat{\mathcal{R}}$ of families such that for each $F \in \widehat\R$ each number $a_n$ is finite and locally constant: for each $n$ there esists $\varepsilon_n$ such that $a_n$ cannot be changed by an $\varepsilon_n$-perturbation of $F$ in the space of families.
\end{lem}

\begin{lem}\label{lem:set_a_n}
Let $F = (f_\theta)$ be a family in which any rotation number from a nondegenerate segment $J \subset [0, 1)$ appears, let $\{q_n\mid n \in \mathbb N\}$ be the enumeration of rational numbers in~$J$, and let $a_n = a_n(F)$ be the corresponding sequence of maximal numbers of hyperbolic sources possessed by diffeomorphisms $f_\theta$ when $\rho(f_\theta) = q_n$.
 Then for any $\varepsilon>0$ there exist integers $n_0, N$ such that for any $m > N$ the $\varepsilon$-neighborhood of~$F$ contains a perturbed family $\tilde F \in \mathcal{\widehat{R}}$ such that $a_{n_0}(\tilde{F})$ equals~$m$.
\end{lem}

In fact, we will refer to the following corollary rather than to the last lemma itself.

\begin{cor}\label{cor:set_parity}
In the assumptions of Lemma~\ref{lem:set_a_n}, for any residual set $\mathcal{R}$ of families and any $\varepsilon>0$, there exists the smallest $n$ such that the parity number $b_n(F)$ can be changed by an $\varepsilon$-perturbation of the family with the perturbed family being in~$\mathcal{R}$.
\end{cor}
\begin{proof}
Lemma~\ref{lem:set_a_n} implies that for a given perturbation size the set of $n$ for which one can find a family $\tilde{F} \in \mathcal{\widehat{R}}$ that is $\varepsilon$-close to $F$ but satisfies $b_n(\tilde{F}) \ne b_n(F)$ is nonempty. Note that since $\tilde F \in \widehat{\mathcal R}$, there exists $\delta>0$ such that $b_n$ is constant in a ball $B_\delta(\tilde F)$. Shrinking $\delta$ if necessary, we may assume that $B_\delta(\tilde F)\subset B_\varepsilon(F)$. Hence, by density of $\mathcal R$, we can replace $\tilde{F}$ by an element of $\mathcal R \cap B_\varepsilon(F)$ with the same~$b_n$. Therefore, the set of $n$ such that there is a family $\tilde{F} \in \mathcal R \cap B_\varepsilon(F)$ with $b_n(\tilde{F}) \ne b_n(F)$ is nonempty and has a minimal element.
\end{proof}
We will always be using Corollary~\ref{cor:set_parity} with $\mathcal{R}$ being the residual set $\widehat{\mathcal{R}}$ of Lemma~\ref{lem:parity_stable}.

\subsection{Mapping binary words into families}
First, we replace $\mathcal R$ by $\mathcal R \cap \widehat{\mathcal{R}}$, where $\widehat{\mathcal{R}}$ comes from Lemma~\ref{lem:parity_stable}.
Then we define a map $\Xi$ from the set of finite binary words to $\mathcal R$ by an inductive procedure. We also inductively define two sequences: the sequence $(\varepsilon_n)_{n\ge 0}$ of allowed perturbation sizes and the sequence of considered indices $(i_n)_{n \ge 0}$. The empty word $\Lambda$ is taken by~$\Xi$ into an arbitrary element $F_\Lambda \in\mathcal R \subset \mathcal{U}_1$. We choose the first element $\varepsilon_0 > 0$ to be so small that a closed $\varepsilon_0$-neighborhood of $F_\Lambda$ is contained in $\mathcal U_1$. The sequence $(i_n)$ is initialized by setting $i_0 = 0$. Assume that
\begin{enumerate}[label=\arabic*)]
  \item the map~$\Xi$ is already defined for all binary words of length up to $n$ and the numbers $\varepsilon_k, i_k$ are defined for $k\le n$ as well;
  \item $\varepsilon_n>0$ is so small that for each word $w$ of length $n$ a closed $\varepsilon_n$-ball centered at $\Xi(w)$ is contained in~$\mathcal{U}_{n+1}$;\footnote{Recall that $\mathcal R = \cap_{n \ge 1}\mathcal U_n$ and all $\mathcal{U}_n$ are open.}
  \item $\varepsilon_n$ is so small that for each word $w$ of length $n$ the parity numbers $b_j$ of the family $\Xi(w)$ cannot be changed by a $2\varepsilon_n$-perturbation of the family for $j = 1,2,\dots, i_n$.
\end{enumerate}

We explain how to define $\Xi$ on the words of length $n+1$.
First, consider the words $\tilde{w}$ of lenth~$n+1$ that end with a zero: $\tilde{w} = w0$, and set $\Xi(\tilde{w}) = \Xi(w)$. Then set $\varepsilon_{n+1} = \varepsilon_n/2$ and $i_{n+1} = i_n + 1$. For each word $\tilde{w} = w1$ of length $n+1$, consider the family $\Xi(w)= F_w \in \mathcal R$. 
By Corollary~\ref{cor:set_parity}, there exists the smallest $j>i_n$ such that $b_j$ can be changed by an $\varepsilon_{n+1}$-perturbation of $F_w$. For future reference, denote this $j$ by $j_w$. Set $c_{w0}$ equal to the parity number $b_j$ of $F_w$ and set $c_{\tilde{w}} = c_{w1} = 1 - b_j$. Choose an $\varepsilon_{n+1}$-perturbed family $F_{\tilde{w}}$ for which the $j$-th parity number equals $c_{\tilde{w}}$ and set $\Xi(\tilde{w}) = F_{\tilde{w}}$. Note that the parity numbers $b_i$, for $i<j_w$, coincide for the families $\Xi(w)$ and $\Xi(\tilde{w})$. If $j>i_{n+1}$, replace $i_{n+1}$ with~$j$. After dealing with all words of length $n+1$, further decrease $\varepsilon_{n+1}$ so that a $2\varepsilon_{n+1}$-neighborhood of any $F_{\tilde{w}}$ is contained in $\mathcal{U}_{n+2}$ and the parity numbers of any $F_{\tilde{w}}$ up to the index $i_{n+1}$ cannot be changed by a $2\varepsilon_{n+1}$-perturbation. Now we have 1)-3) for the words of length $n+1$ and we can proceed with the words of length $n+2$, etc.

\subsection{The limit families}
Consider a binary sequence $\omega \in \{0, 1\}^{\mathbb N}$ viewed as a nested sequence of finite binary words~$w_n$. With a slight abuse of notation, we will write $\omega = \cup_{n=1}^{+\infty} w_n$.
\begin{prop}
 For any binary sequence $\omega = \cup_{n=1}^{+\infty} w_n$, the sequence of families $F_{w_n} = \Xi(w_n)$ converges to some family $F_{\omega} \in \mathcal{R}$.
\end{prop}
\begin{proof}
The construction implies that $\dist(F_{w_{n+1}}, F_{w_n}) \le \frac{\varepsilon_n}{2} \le \frac{\varepsilon_{n-1}}{4} \le \dots  \le \frac{\varepsilon_0}{2^n}$, hence $(F_{w_n})$ is a Cauchy sequence. All spaces of families that we consider are assumed to be endowed with a metric that makes them complete, so the sequence $F_{w_n}$ converges to some family that we denote~$F_\omega$. A closer look reveals that for $m > n$ we have an estimate
\begin{equation}\label{eq:dist_nm}
  \dist(F_{w_n}, F_{w_m}) \le \frac{1}{2}(\varepsilon_n + \varepsilon_{n+1} + \dots + \varepsilon_{m-1})
  \le \varepsilon_{n}\cdot \left(\frac{1}{2} + \frac{1}{4} + \dots + \frac{1}{2^{m-n}}\right) < \varepsilon_n.
\end{equation}
Recall that $\varepsilon_n$ was chosen so that a closed $\varepsilon_n$-ball centered at $F_{w_n} = \Xi(w_n)$ is contained in~$\mathcal U_{n+1}$. Hence, for each $n$, the tail sequence $(F_{w_j})_{j \ge n}$ is contained in a closed $\varepsilon_n$-ball in $\mathcal{U}_{n+1}$. The limit family $F_\omega$ belongs to all $\mathcal{U}_n$, and hence to $\mathcal{R} = \cap_n \mathcal{U}_n$.
\end{proof}

\begin{prop}\label{prop:1st_dif}
For the limit family $F_\omega$, for any $n \ge 1$ the parity sequence $(b_j)$ coincides with that of $F_{w_j}$ up to the index $i_n$.
\end{prop}
\begin{proof}
Inequality~\eqref{eq:dist_nm} implies $\dist(F_{w_n}, F_{w_m}) \le \varepsilon_n$, so the choice of $\varepsilon_n$  guarantees that the parity sequence numbers $b_j$ with $j = 1, \dots, i_n$ coincide for $F_{w_n}$ and $F_\omega$.
\end{proof}

\begin{prop}\label{prop:distinct}
For distinct binary sequences $\omega$ and $\tilde{\omega}$, the limit families $F_{\omega}$ and $F_{\tilde\omega}$ belong to different classes of weak equivalence.
\end{prop}
\begin{proof}
Let $n$ be the first index where $\omega = \cup_{n=1}^{+\infty} w_n$ and $\tilde\omega = \cup_{n=1}^{+\infty} \tilde{w}_n$ differ. Then by construction the parity sequences of the families $F_{w_n}$ and $F_{\tilde{w}_n}$ first differ at the index $j_{w_{n-1}} \le i_n$. By Proposition~\ref{prop:1st_dif}, the parity sequences for the limit families $F_{\omega}$ and $F_{\tilde\omega}$ differ at the same index. Since the parity sequence is an invariant of weak equivalence, the two limit families are not weakly equivalent.
\end{proof}
Prorosition~\ref{prop:distinct} finishes the proof of Theorem~\ref{thm:continuum} modulo Lemmas~\ref{lem:set_a_n} and~\ref{lem:parity_stable}.

\subsection{Proof of Lemma~\ref{lem:parity_stable}}
\label{sec:p_lemma}

Recall that for a family $F$ we denote by $a_n$ or $a_n(F)$ the maximum number of sources that can coexist in $F$ for the same parameter value at which the rotation number is~$q_n \in J \cap \mathbb{Q}$.
Let us fix $n$ and prove that there exists an open and dense set $\widehat\U_n$ of families for which $a_n$ is finite and locally constant. Then we will set $\widehat\R = \cap_n \widehat\U_n$.

Let $q_n = \frac{p}{q}$ with $p, q$ coprime. Then all periodic orbits we are interested in will have period~$q$.
By the results of~\cite{B1},\footnote{More precisely, we refer to the sets $\tilde{\mathcal{F}}_{j,l}$ from the proof of Theorem~1 of~\cite{B1} being open and dense. One must note that for each family $(f_\theta)$ of circle diffeomorphisms the locus of periodic points of period less than $q$ is separated from the set $\{\theta\colon \rho(f_\theta)\cdot q \in \mathbb Z)\} \times S^1$ and then take $j = q$ and $l$ sufficiently large.} for any $2 \le r < +\infty$, there is an open and dense set of $C^r$-families of circle diffeomorphisms with the property that non-Morse-Smale diffeomorphisms with periodic points of period $q$ appear at isolated parameter values and display a single nondegenerate saddle-node bifurcation orbit; in particular, all sources in these families are hyperbolic. We have the same property densely in the space of $C^\infty$-families\footnote{The $C^\infty$-density is easy to obtain. By the previous $C^r$-result about openness and density, each $C^\infty$ family $F$ can be $1/r$-approximated in $C^r$ by a $C^\infty$-family $F_r$ with the required property. The sequence $(F_r)$ converges to $F$ in each $C^n$, hence also in $C^\infty$. In fact, the property is also open in~$C^\infty$, see Remark~\ref{rem:nondenfam} below.} and, in particular, in the space of families that we consider.
In what follows, when we refer to \emph{a non-degenerate} one-parameter family, we assume that it has this property.

Consider a non-degenerate one-parameter family.
We can also assume that the endpoints of the parameter segment $I$ correspond to structurally stable circle diffeomorphisms. It is easy to see that for $F$ the value of $a_n$ is finite: otherwise we consider the sequence of parameter values $(\theta_m)$ over which the number of sources goes to infinity, take a subsequence limit $\theta_\infty$ and argue that $f_{\theta_\infty}$ either has infinitely many periodic orbits of period $q$ or finitely many infinitely degenerate periodic orbits, but any of that is a contradiction.

Since all sources in the family $F$ are hyperbolic, it is obvious that $a_n$ cannot decrease in a small neighborhood of $F$: if, say, $f_{\theta_0}$ has $a_n$ sourses, the same is true for a nearby family for the same parameter value, since hyperbolic sources survive small perturbations of the diffeomorphism. Hence, $a_n(F) \le a_n(\tilde F)$ for $\tilde F$ near~$F$.

Assume that $a_n$ can increase in any neighborhood of~$F$. Then there is a sequence of families $F_m \to F$ such that $a_n(F_m) > a_n(F)$. Denote $N = a_n(F)$. Without loss of generality, we can assume that each $F_m$ has at least $N+1$ orbits of hyperbolic sources for some value of~$\theta$ (a priori those orbits are not hyperbolic, but yield hyperbolic orbits after a small perturbation of the diffeomorphism). After replacing $(F_m)_{m\in\mathbb{N}}$ by its subsequence, we find a convergent sequence $\theta_m \to \theta_\infty$ of parameter values such that $F_m$ has at least $N+1$ orbits of hyperbolic sources at parameter value~$\theta_m$.  Choose any $N+1$ such orbits for each $F_m(\theta_m)$ and denote them by $s_1(m), \dots, s_{N+1}(m)$. Note that $F(\theta_\infty)$ either has no parabolic orbits and at most $N$ orbits of hyperbolic sources, or one parabolic orbit and at most $N-1$ orbits of hyperbolic sources. After replacing $(F_m)_{m\in\mathbb{N}}$ with its subsequence again, we may assume that each $s_j(m)$ converges (as a compact set of~$q$ points) to a periodic orbit $s_j$ of the map $F(\theta_\infty)$. Obviously, $s_j$ cannot be an orbit of a hyperbolic sink, so it is either a saddle-node orbit or a source orbit. But each such orbit can yield at most one hyperbolic source after a small perturbation of~$F(\theta_\infty)$. Indeed, if $s_j$ is a hyperbolic source orbit, any perturbation of $F(\theta_\infty)$ has exactly one periodic orbit in a small neighborhood of $s_j$. The same is true for a nondegenerate saddle-node orbit: if a nearby diffeomorphism had two sources near a saddle-node, there would also be another periodic point between the sources for topological reasons, so applying Rolle's theorem twice and using the continuity of the second derivative we get $(f_{\theta_\infty}^{q})''(x) = 0$ at the saddle-node, that is, the saddle-node must be degenerate.
%THE ROLLE THEOREM ARGUMENT SPELLED OUT:
% Let $g_\theta(x)=f_\theta^q(x)-x$. Suppose that a nondegenerate parabolic orbit of $g_{\theta_\infty}$ is approximated by two source orbits of $g_{\theta_m}$. Let $x_m^1<x_m^2$ be two fixed points of $g_{\theta_m}$ corresponding to these sources, and assume $x_m^1,x_m^2\to x_0$, where $x_0$ is the parabolic fixed point of $g_{\theta_\infty}$.

%Since both fixed points are sources, the sign of $g_{\theta_m}$ changes from negative to positive at each of them. Hence there must be at least one further fixed point $y_m\in(x_m^1,x_m^2)$ between them. Applying Rolle's theorem to $g_{\theta_m}$ on the intervals $[x_m^1,y_m]$ and $[y_m,x_m^2]$, we obtain two points $u_m<v_m$ such that $$g_{\theta_m}'(u_m)=g_{\theta_m}'(v_m)=0.$$ Applying Rolle's theorem once more to $g_{\theta_m}'$ on $[u_m,v_m]$, we get a point $w_m\in(u_m,v_m)$ with $$g_{\theta_m}''(w_m)=0.$$ Passing to a subsequence, we may assume $w_m\to x_0$. By continuity of the second derivative, this implies $$g_{\theta_\infty}''(x_0)=0,$$contradicting the nondegeneracy of the parabolic orbit.
So, there is a contradiction: $N+1$ source orbits must accumilate to at most $N$ orbits, but each of those can be accumulated by at most one source orbit. The proof of Lemma~\ref{lem:parity_stable}
is complete.

\medskip

\begin{rem*}
Here is a good place to admit that for $r=1$ Lemma~\ref{lem:parity_stable} does not hold and our approach to proving Theorem~\ref{thm:continuum} does not work. By $C^1$-perturbing a diffeomorphism with a parabolic point, one can get a diffeomorphism with a segment of periodic points. One can further perturb this latter diffeomorphism to encode an infinite sequence of $0$s and $1$s in such a way that it becomes an invariant of topological classification of families. The encoding may take the form of a sequence of hyperbolic periodic attractors (zeros) and repellors (ones) with basins separated by segments of degenerate periodic points. Given two families, one can check if it is true that each of them contains a single diffeomorphism that bares such a mark and, if yes, compare the marks. For any given mark, any open set of $C^1$-families will contain a family with a diffeomorphism bearing this mark (and no other diffeomorphisms with infinitely many periodic orbits), hence any open set will have a continuum of equivalence classes. On the other hand, for a generic $C^1$-family all $a_n(F)$ are equal to $+\infty$: the sets $\mathcal{U}_{n,j}$ of families where there are at least $j$ hyperbolic sources for some map with rotation number $q_n$ are all open and dense in~$C^1$.
\end{rem*}

\subsection{Proof of Lemma~\ref{lem:set_a_n}}
\label{sec:s_lemma}

Recall that Lemma~\ref{lem:set_a_n} says that given a family with nonconstant rotation number, we can set up some $a_{n_0}$ by an $\varepsilon$-small perturbation of the family.

By Lemma~\ref{lem:pert}, we can increase some $a_{n_0}$ by $\varepsilon$-perturbing the family~$F$, but this lemma does not give us control over the actual value of $a_{n_0}$ for the perturbed family~$\widehat{F}$. The idea of the proof is to connect $F$ and $\widehat{F}$ by a smooth homotopy, perturb this homotopy so that it is a two-parameter family with generic properties and then argue that when we change the homotopy parameter the number $a_{n_0}$ changes in increments of~$1$. The properties that we need from the two-parameter family are described in the following lemma, the proof of which is postponed to Section~\ref{sec:transv}. For two-parameter families, we assume that $B = I_s \times I_\theta$, where $I_s = [0,1],\; I_\theta = [-1,1]$, and $s$ is viewed as the homotopy parameter. The parameter space becomes a manifold with corners, but related technical difficulties can be evaded by, say, embedding $I_s, I_\theta$ into circles or arguing that there is an open and dense set of families where the corners correspond to structurally stable circle diffeomorphisms. We write
\[F \colon I_s \times I_\theta \times S^1 \to S^1,\; I_s = [0,1],\; I_\theta = [-1,1], \quad (s, \theta, x) \mapsto F_{(s,\theta)}(x).\]
In the parameter space, we refer to the $s$-axis  as vertical and the $\theta$-axis as horizontal. 

\begin{lem}\label{lem:trans}
In the space of $C^r$-smooth ($1\le r\le \infty$) two-parameter families, there is a dense set of families with the following properties.
\begin{enumerate}
  \item In the parameter space, the subset that corresponds to diffeomorphisms with rotation number $q_n$ with a nonhyperbolic periodic orbit is a union of finitely many $C^2$-smooth regular curves $\gamma_j, \; j = 1,\dots, N,$ and points $c_i, \; i = 1, \dots, M$.
  \item The curves $\gamma_j$ have endpoints either at one of $c_i$ or at the boundary of~$I_s\times I_\theta$, excluding the corners. The curves are transverse to the boundary. The points of the curves, except for~$c_i$, correspond to diffeomorphisms with quadratic saddle-node bifurcations; the points $c_i$ correspond to cusp bifurcations. The points $c_i$ are in the interior of~$I_s\times I_\theta$.
  \item The curves $\gamma_j$ have only transverse intersections. Any point of intersection is in $\operatorname{Int}(I_s\times I_\theta)$ and has exactly 2 preimages in the union of the domains of $\gamma_j.$\footnote{That is, self-intersections are allowed, but 3 pieces of curves cannot intersect at one point.} There are only finitely many points where the curves are tangent to the horizontal $s = const$ direction.
\end{enumerate}
These properties are generic for $r \ge 4$.
\end{lem}

\begin{proof}[Proof of Lemma~\ref{lem:set_a_n}]
Consider a one-parameter family as in the statement of Lemma~\ref{lem:set_a_n}. Approximate it, via an $\varepsilon/4$-perturbation, by a family from the residual set $\widehat{\mathcal{R}}$ of Lemma~\ref{lem:parity_stable}.  Replace the original~$F$ by the perturbed one and set $N = a_{n_0}(F)$. Apply Lemma~\ref{lem:pert} to obtain a new perturbed family $\widehat{F}$ such that the number $a_{n_0}(\widehat{F})$ is larger than $m > N$. Obviously, one can take $\widehat{F} \in \widehat{\mathcal R}$. Connect $F$ and $\widehat{F}$ with a smooth homotopy $F_{(s,\theta)}$ such that $F_{(0,\cdot)} = F$, $F_{(1,\cdot)} = \widehat{F}$ and any $F_{(s,\cdot)}$ is $\varepsilon/2$-close to the original $F$ in the relevant space of one-parameter families. Perturb the homotopy map to obtain a $C^\infty$-smooth two-parameter family as in Lemma~\ref{lem:trans}. Assume that after this perturbation $a_{n_0}(F_{(0,\cdot)})$ and $a_{n_0}(F_{(1,\cdot)})$ do not change and the intermediate families $F_{(s,\cdot)}$ remain $\varepsilon/2$-close to the original~$F$.

Consider the curves $\gamma_j$ and the cusp-bifurcation parameter points $c_i$ in the parameter space $I_s \times I_\theta$. By Lemma~\ref{lem:trans}, there is a finite set  $\Omega = \{h_i\}$ of points where $\gamma_j$ have horizontal tangent lines, a finite set~$X$ where the curves $\gamma_j$ intersect, and the set of cusp bifurcation parameter values ${\mathcal V} = \{c_i\}_{i = 1}^M$ is also finite. Also consider the finite set $Y$ of transverse intersections between the curves $\gamma_j$ and the boundary $\partial B$. If there is a section $\{(s, \theta)\colon s = s_0\}$ that contains more than one point of $\Omega\cup X \cup {\mathcal V} \cup Y$, we can make a coordinate change in the parameter space, arbitrarily close to the identity, after which this is not the case. Such a coordinate change corresponds to a small perturbation of the family.

The complement $A$ to the set $\bigcup_j\gamma_j \cup {\mathcal V}$ in the parameter space is an open set in the topology of the parameter rectangle. Consider a connected component $U$ of this set. Either for $U$ every corresponding diffeomorphism has rotation number different from $q_n$ or all diffeomorphisms have rotation number $q_n$, are structurally stable and have the same number of sources. Indeed, $U$ is pathwise-connected, so if it contained points that correspond to rotation number $q_n$ and some other rotation number, one would connect these points by a curve inside $U$ and find a point on this curve where the rotation number is $q_n$ and a nonhyperbolic $q$-periodic orbit is present.

Consider a horizontal section $\{s = const\}$ that contains no points of~$\Omega \cup X \cup {\mathcal V} \cup Y$. Then any section nearby corresponds to a family $F_{(s,\cdot)}$ with the same value of~$a_{n_0}$: this section intersects the same connected components of~$A$.

Now consider a section $\{s = s_0\}$ that contains exactly one point of~$\Omega \cup X \cup {\mathcal V} \cup Y$. When the horizontal section crosses such a point, the value of $a_{n_0}$ can change by at most~1. Indeed, for a point in $\Omega$, $Y$, or $\mathcal{V}$, its small neighborhood intersects two connected components of $A$ and the number of sources in these differs by at most~1. In the case of a point $p \in X$, the two intersecting curves $\gamma_j$ (or the pieces of one self-intersecting curve) split a neighborhood of~$p$ into 4 regions. Since at least one of the intersecting curves $\gamma_j$ is transverse to $\{s = s_0\}$ at~$p$, for small $\delta > 0$ the sections $\{s = s_0\pm\delta\}$ intersect at least 2 adjacent regions of 4. On the other hand, the number of sources changes by at most 1 between the adjacent regions. Hence, for the sections $\{s = s_0+\delta\}$ the number $a_{n_0}$ can change by at most 1 compared to the sections $\{s = s_0-\delta\}$. But then for any integer $m$ between $N$ and $a_{n_0}(\widehat{F})$ we can find a value of $s$ such that $a_{n_0}(F_{(s,\cdot)})=m$. The proof of Lemma~\ref{lem:set_a_n} modulo Lemma~\ref{lem:trans} is complete.
\end{proof}

\section{Proof of Lemma~\ref{lem:trans}}\label{sec:transv}

\subsection{Multijet transversality}

The proof of Lemma~\ref{lem:trans} can be reduced to a chain of informal arguments like this: ``For a circle diffeomorphism, to have simultaneously a cusp bifurcation periodic orbit and another parabolic periodic orbit is a degeneracy of codimension 3, so such maps do not appear in generic 2-parameter families''. Such arguments are formalized with the help of the multijet transversality theorem, the $C^\infty$-version of which is attributed to C. Morlet in the work of J.~N.~Mather~\cite{Ma}. G.~K.~Francis explained in~\cite{F} why the finitely-smooth version of this theorem holds, and P.W. Michors book~\cite{Mi} contains a version for manifolds with corners.

For an integer $k\ge 0$ and smooth manifolds $X,Y$, denote by $J^k(X,Y)$ the bundle of $k$-jets of maps in $C^k(X,Y)$. For $m\ge 1$, the space of $m$-multi-$k$-jets consists of tuples in $(J^k(X,\,Y))^m$ with elements that have pairwise distinct projections to~$X$. We denote this space ${}_mJ^k(X,Y)$.
Every map $F\in C^r(X, Y)$ has its $m$-multijet extension
\[
{}_mj^r F:
(x_1, \dots, x_m) \mapsto (j^r_{x_1}F, \dots, j^r_{x_m}F),
\]
where $x_j$ are assumed to be pairwise distinct.

\begin{theorem*} (Multijet transversality,~\cite{Ma,Mi,F}). 
Let $W$ be a submanifold in the space of $m$-multi-$k$-jets ${}_mJ^k(X,Y)$ of codimension $c$ and 
$r - k > \max(0, mn - c)$, where $n = \dim X$ and $\dim Y < \infty$. 
Then the set of $f \in C^r(X,Y)$ for which ${}_mj^k f \pitchfork W$ is residual.
\end{theorem*}

The application of this theorem is complicated by the fact that we are working with iterations of a mapping, therefore we first introduce the relevant submanifolds in the space if multijets and deal with the case of fixed parabolic points in Section~\ref{sec:submanifolds} and then explain in Section~\ref{sec:fix_to_per} how the iterates of the family inherit the transversality properties.

\subsection{Submanifolds in the space of multijets}
\label{sec:submanifolds}

Let $B=I_s\times I_\theta$ be the parameter rectangle. We will denote its elements by $\beta = (s, \theta)$ for brevity. Let $U\subset B$ be a closed ball, $A\subset\mathbb S^1$ be a finite disjoint union of segments, and $E \subset S^1$ be a finite union of intervals such that $A\cap E\neq\varnothing$. Denote
\[
\mathfrak C^r \;:=\; C^r\big(U\times A,\; E\big).
\]
In the definitions below we can assume that $r \ge 3$, but in the relevant arguments we will need the assumption $r \ge 4$.
For a map $F\in\mathfrak C^r$, its $m$-multi-$k$-jet ($k \le r$) extension has the form
\[
{}_mj^k F:
((\beta, x_1), \dots, (\beta_m,x_m)) \mapsto (j^k_{(\beta_1,x_1)}F, \dots, j^k_{(\beta_m,x_m)}F),
\]
where $(\beta_j, x_j)$ are distinct.

The dimensions of the domain and the relevant (multi)jet spaces are as follows.
Since $\dim(U\times A)=3$, the dimension of the $k$-jet bundle for essentially scalar-valued
maps is
\[
\dim J^k(U\times A,\,E)
= \dim(U\times A) + \binom{\dim(U\times A)+k}{k}.
\]
Hence
\[
\dim J^1(U\times A,\,E) = 7,  \quad
\dim J^2(U\times A,\,E) = 13, \quad
\dim J^3(U\times A,\,E) = 23.
\]
For multijet spaces we have
\[
\dim {}_mJ^k(U\times A,\,E)
= m \cdot \dim J^k(U\times A,\,E),
\]
while the domain of the $m$-multijet extension has dimension $m \cdot \dim(U\times A) = 3m$.

\medskip
Now we define the submanifolds of the jet and multijet spaces that we need to consider to prove Lemma~\ref{lem:trans}. The partial derivatives of different order will be denoted by $\partial_x F, \; \partial_{xx} F, \; \partial_{x\theta} F,$ etc.

\begin{enumerate}
  \item The submanifold $\Sigma_{\mathrm{par}}$ of \textbf{parabolic fixed points}\footnote{Of course, $\Sigma_{\mathrm{par}}$ contains the $1$-jets of $F$ at parabolic fixed points rather than the points themselves, but we hope that the reader would find this language abuse appropriate.} is 
  a subset of $J^1(U\times A,\,E)$ defined as
  \[
  \Sigma_{\mathrm{par}}
  \;=\;
  \{\,j^1_{(\beta,x)}F:\;
  F_\beta(x)-x=0,\; \partial_xF_\beta(x)=1\,\}.
  \]
  This is a codimension 2 submanifold in the space of 1-jets that correspond to fixed points of $F_\beta(\cdot)$ with multiplier~$1$.
  
  \item The set \[
  \Sigma_{\mathrm{par}}^{(2)} =
  \Big\{\left(j^1_{(\beta,x_1)}F,\,j^1_{(\beta,x_2)}F\right):\ 
    F_\beta(x_i)=x_i,\ \partial_xF_\beta(x_i)=1,\ i=1,2;\ x_1 \ne x_2
  \Big\}
  \]
  is the submanifold of \textbf{two coexisting parabolic fixed points}. That is, we consider the case when two parabolic fixed points occur simultaneously for the same parameter~$\beta$. This submanifold has codimension~$6$ in ${}_2J^1(U\times A,\,E)$: the condition $\beta = \beta_1 = \beta_2$ that adds 2 to the codimension is hidden in the notation. 
  
  \item The set $\Sigma_{\mathrm{par}}^{(3)}$ of  \textbf{three coexisting parabolic fixed points} is defined analogously to $\Sigma_{\mathrm{par}}^{(2)}$ and has codimension $10$ in ${}_3J^1(U\times A,\,E)$.
  
  \item The submanifold of \textbf{degenerate parabolic fixed points}
  \[
  \Sigma_{\mathrm{deg}}
  \;=\;
  \{\,j^2_{(\beta,x)}F:\; F_\beta(x)=x,\; \partial_xF_\beta(x)=1,\; \partial_{xx}F_\beta(x)=0\,\}.
  \]
  is a subset of \(J^2(U\times A,\,E)\) of codimension $3$. It contains an open subset $\Sigma_{\mathrm{cusp}}$ of cusp-bifurcation fixed points distinguished (see~\cite{K}) by conditions $\partial_{xxx}F_\beta(x)\ne0$ and
  \begin{equation}\label{eq:cusp_cond}
  \left(\partial_sF\cdot\partial_{x\theta}F - \partial_\theta F\cdot\partial_{xs}F\right)(\beta, x) \ne 0.
  \end{equation}
  We define the sets $\Sigma_{\mathrm{deg, 1}}, \; \Sigma_{\mathrm{deg, 2}} \subset J^3(U\times A,\,E)$ where these conditions are violated, in order to argue later that such fixed points are irrelevant:
  \[\Sigma_{\mathrm{deg, 1}} = \{\,j^3_{(\beta,x)}F \colon j^2_{(\beta,x)}F \in \Sigma_{\mathrm{deg}}, \, \partial_{xxx}F_\beta(x)=0\};\]
  \[\Sigma_{\mathrm{deg, 2}} = \{\,j^3_{(\beta,x)}F \colon j^2_{(\beta,x)}F \in \Sigma_{\mathrm{deg}}, \, \left(\partial_sF\cdot\partial_{x\theta}F - \partial_\theta F\cdot\partial_{xs}F\right)(\beta, x) = 0\}.\]
  Both have codimension 4 in $J^3(U\times A,\,E)$. More precisely, $\Sigma_{\mathrm{deg, 2}}$ has a singular locus 
  \[
  \{\partial_sF(\beta, x) = \partial_{x\theta}F(\beta, x) = \partial_\theta F(\beta, x) = \partial_{xs}F(\beta, x) = 0\}
  \]
  and should be viewed as a stratified submanifold with 4 being the smallest codimension of stratum.

  \item The \textbf{submanifold of coexisting degenerate and arbitrary parabolic fixed points}
  \[\Sigma^{(2)}_{\mathrm{deg, par}} = 
  \Big\{\left(j^2_{(\beta_1,x_1)}F,\,j^2_{(\beta_2,x_2)}F\right) \colon \beta_1 = \beta_2,\, 
    j^2_{(\beta_1,x_1)}F \in \Sigma_{\mathrm{deg}},\, j^1_{(\beta_2,x_2)}F \in \Sigma_{\mathrm{par}}
  \Big\}
  \]
  has codimension 7 in ${}_2J^{2}(U \times A,\,E)$.
  
  \item The \textbf{submanifold related to two coexisting parabolic points that yield a nontransverse intersection in the bifurcation diagram}
  \[\Sigma^{(2)}_{\mathrm{par, tan}} = 
  \Big\{\left(j^1_{(\beta,x_1)}F,\,j^1_{(\beta,x_2)}F\right) \in \Sigma_{\mathrm{par}}^{(2)}: \partial_\theta F_\beta(x_1)\partial_sF_\beta(x_2) - \partial_\theta F_\beta(x_2)\partial_sF_\beta(x_1) = 0 
  \Big\}\]
  is a stratified manifold of minimal codimension 7 in ${}_2J^1(U\times A,\,E)$.
  \item The \textbf{submanifold related to `horizontal' tangent lines in the bifurcation diagram}
  \[\Sigma_{\mathrm{par,hor}} = 
  \Big\{j^1_{(\beta,x)}F \in \Sigma_{\mathrm{par}} \colon\; \partial_{\theta}F_\beta(x) = 0 
  \Big\}\]
  is a submanifold of codimension 3 in $J^1(U\times A,\,E)$. This submanifold contains the jets of $F$ at all non-degenerate parabolic fixed points $(\beta, x)$ such that the tangent vector at $\beta$ to the line of parabolic fixed points satisfies $ds = 0$. It may also contain the jets at some degenerate parabolic points, but we do not care. Also, at some point in the argument we will need an analogously defined submanifold related to vertical tangent lines:
  \[\Sigma_{\mathrm{par,vert}} = 
  \Big\{j^1_{(\beta,x)}F \in \Sigma_{\mathrm{par}} \colon\; \partial_{s}F_\beta(x) = 0 
  \Big\}.\]
\end{enumerate}
\begin{rem}
    Condition~\eqref{eq:cusp_cond} appears when checking, via the implicit function theorem, that the preimage $W_{\mathrm{par}}$ of $\Sigma_{\mathrm{par}}$ in $I_s\times I_\theta \times A$ under the jet-extension of $F$ is a regular curve near a cusp fixed point. The cusp itself appears when we project this preimage to $B = I_s\times I_\theta$.
\end{rem}

\begin{rem}
    The condition in the definition of $\Sigma^{(2)}_{\mathrm{par, tan}}$ is obtained in the following way. We project the preimage $W_{\mathrm{par}}$ to $I_s\times I_\theta$ and consider a point where the projection intersects itself (if there are any). Let this intersection correspond to points $x_1, x_2 \in A$ and parameter value $\beta = (s,\theta)$. Then the intersection is transverse iff\footnote{If a curve in $\mathbb R^3$ is given by a system of equations $H=0$ with $H = (H_1, H_2)$, its tangent vector $v(q)$ at a point $q$ is given by $v(q) = \nabla H_1(q) \times \nabla H_2(q)$. In our case, $F_\beta(x)-x$ plays the role of $H_1$ and $\partial_xF_\beta(x)-1$ is $H_2$. We calculate the tangent vectors $v(q)$ at $q = (\beta, x_1)$ and $q = (\beta, x_2)$ and write the transversality condition
    \[\det
    \begin{pmatrix}
    ds(v(\beta, x_1))& d\theta(v(\beta, x_1))\\
    ds(v(\beta, x_2))& d\theta(v(\beta, x_2))
    \end{pmatrix}\ne 0.\]}
\begin{equation}\label{eq:projection}
      \partial_{xx}F_\beta(x_1)\cdot\partial_{xx}F_\beta(x_2)\cdot\left(\partial_\theta F_\beta(x_2)\partial_sF_\beta(x_1) - \partial_\theta F_\beta(x_1)\partial_sF_\beta(x_2)\right) \ne 0.
\end{equation}
    Since we already have $\Sigma_{\mathrm{deg, par}}$ to exclude degenerate parabolic points coexisting with other parabolic points, we use only the 3rd factor of~\eqref{eq:projection} in the definition of~$\Sigma^{(2)}_{\mathrm{par,tan}}$.
    
Likewise, the condition that the projection of $W_{\mathrm {par}}$ to the parameter space is not tangent to $ds = 0$ direction at $\beta$ can be written as $ds(v(\beta, x)) \ne 0$, or, equivalently, $\partial_{xx} F_\beta(x)\cdot \partial_\theta F_\beta(x) \ne 0,$ where $(\beta, x) \in W_{\mathrm{par}}$. 
\end{rem}

\medskip

\begin{prop}[Generic multijet transversality for the listed degeneracies]\label{prop:multijet-transversality}
Let \(r\ge 4\) and let \(\mathfrak C^r=C^r(U\times A,\,E)\). Then a topologically generic $F \in \mathfrak{C}^r$ has (multi)jet extensions transverse to the (stratified) submanifolds 
\[
\Sigma_{\mathrm{par}},\ 
\Sigma_{\mathrm{par}}^{(2)},\ 
\Sigma_{\mathrm{par}}^{(3)},\ 
\Sigma_{\mathrm{deg}},\ 
\Sigma_{\mathrm{deg,1}},\
\Sigma_{\mathrm{deg,2}},\ 
\Sigma_{\mathrm{deg,par}}^{(2)},\
\Sigma_{\mathrm{par,tan}}^{(2)},\
\Sigma_{\mathrm{par,hor}}.
\]
In particular, the jet extensions of a generic $F$ have no intersections with
\[
\Sigma_{\mathrm{par}}^{(3)},\  
\Sigma_{\mathrm{deg,1}},\
\Sigma_{\mathrm{deg,2}},\ 
\Sigma_{\mathrm{deg,par}}^{(2)},\
\Sigma_{\mathrm{par,tan}}^{(2)}
\]
and intersect
$
\Sigma_{\mathrm{par}}^{(2)},\
\Sigma_{\mathrm{deg}},\
\Sigma_{\mathrm{par, hor}}
$
at isolated points.
\end{prop}
\begin{proof}
    The proof is by a straightforward application of the multijet transversality theorem. In the case of a stratified submanifold, we apply the theorem to each stratum. For example, consider the case of $\Sigma_{\mathrm{deg,par}}^{(2)}$. It is a subset of ${}_2J^2(U\times A,\,E)$ of codimension~$7$. The domain of the $2$-multi-$2$-jet extension of $F$ has dimension $6$, which is less than $\operatorname{codim}\Sigma_{\mathrm{deg,par}}^{(2)}$. The minimal smoothness required in the theorem is $3<r$. Hence, for generic $F$, the extension ${}_2j^2F$ must intersect $\Sigma_{\mathrm{deg,par}}^{(2)}$ transversely, i.e., there must be no intersections. The argument for the rest of the claims is analogous. Since a finite intersection of residual sets is residual, a generic $F$ has all required properties simultaneously.
\end{proof}

\subsection{The iterates inherit the properties of generic maps locally}\label{sec:fix_to_per}

Let $B$ be the space of parameters (the parameter base) and let $M$ be the phase space. In this section $M$ can be an arbitrary smooth manifold.
For $r\ge 1$, a $C^r$-family with base $B$ of maps from $M$ to itself is an element of 
$C^r(B\times M,M)$. Consider a family
\[
F \colon B \times M \to M, 
\qquad 
F(\beta,x)=F_\beta(x).
\]
For $j\in \mathbb{Z}$ we denote by $F^j$ the family of $j$-iterates:
\[
F^j(\beta,\cdot)=(F_\beta(\cdot))^{\circ j}.
\]
Consider a tuple of subsets $V_j,\, j=0,\dots, q$ of the phase space~$M$. To see which role these are going to play, one may imagine that for some $\beta_0\in B$ the map $F_{\beta_0}$ has a periodic point 
$x_0$ of (minimal) period~$q$, and $V_j$ are the closed balls\footnote{A closed ball is a subset of~$M$ that becomes a closed ball in some chart.} around $F_{\beta_0}^j(x_0)$ such that $F_{\beta_0}$ takes each ball inside the next one. To prove Lemma~\ref{lem:trans}, we will need to be able to track several periodic points at once, so we allow each $V_j$ to be a finite union of closed balls, or rather simply a compact closure of an open set.

Assume that there exist a closed ball $U$ of $\beta_0$ and finite unions of closed balls 
$V_j,\, j=0,\dots,q$, such that

\begin{enumerate}[label=\arabic*)]
  \item $V_0,\dots,V_{q-1}$ are pairwise disjoint;
  \item for every $\beta\in U$ and every $j=0,\dots,q-1$ we have
  \[
  F_\beta(V_j)\subset\operatorname{int}V_{j+1},
  \]
  and $F_\beta$ is a $C^r$-diffeomorphism from a small neighbourhood of $V_j$ onto its image in $\operatorname{int}V_{j+1}$.
\end{enumerate}

\begin{prop}\label{prop:gen_iter}
For a family $F$ as above, there exists a neighbourhood 
${\mathcal O\ni F}$ in $C^r(B\times M,M)$ such that for any residual subset 
$\mathcal R\subset C^r(U\times V_0,\operatorname{int}V_q)$ the set
\[
\widetilde{\mathcal R}
=
\{\widehat F\in C^r(B\times M,M)\colon 
\widehat F^q|_{U\times V_0}\in\mathcal R\}
\]
is residual in $\mathcal O$.
\end{prop}

\begin{proof}
Let
\[
\mathcal X_j=C^r(U\times V_j,V_{j+1}), 
\qquad 
\mathcal Y=C^r(U\times V_0,V_q).
\]
The restriction map
\[
R:\widehat F\longmapsto 
(\widehat F|_{U\times V_0},\dots,\widehat F|_{U\times V_{q-1}})
\]
is continuous from $C^r(B\times M,M)$ to $\prod_{j=0}^{q-1}\mathcal X_j$.
Since all domains are compact, the maps in $\mathcal X_j$ are proper, and hence composition defines a continuous map
\[
C:\prod_{j=0}^{q-1}\mathcal X_j\to\mathcal Y,
\qquad
C(\phi_0,\dots,\phi_{q-1})
=\phi_{q-1}\circ\cdots\circ\phi_0,
\]
and the $q$-iterate restriction satisfies
\[
\Phi(\widehat F):=\widehat F^q|_{U\times V_0}
= C(R(\widehat F)).
\]

Fix a family $\widehat{F}$ near~$F$. Since each $\widehat{F}_\beta|_{V_j}:V_j\to V_{j+1}$ is a local 
diffeomorphism and the sets $V_0,\dots,V_{q-1}$ are disjoint, the partial composition 
\[
\Psi:=\widehat{F}^{q-1}|_{U\times V_0}
\]
is a $C^r$-diffeomorphism from $U\times V_0$ onto its image in $U\times V_{q-1}$. Hence precomposition with $\Psi$ is a homeomorphism
\[
C^r(U\times V_{q-1},\operatorname{int}V_q)
\rightarrow
C^r(U\times V_0, \operatorname{int}V_q).
\]
Therefore small perturbations of $\Phi(\widehat{F})$ can be realized by perturbing only the last factor 
$\widehat F|_{U\times V_{q-1}}$, i.e., by perturbing $\widehat F$ 
inside $U\times V_{q-1}$. 
This shows that $\Phi$ is locally surjective at $\widehat{F}$: 
there exists a neighbourhood $\mathcal O$ of $F$ such that 
$\Phi(\mathcal O)$ contains a neighbourhood of $\Phi(F)$ in $\mathcal Y$, and the same applies to any $\widehat{F} \in \mathcal{O}$ provided that $\mathcal{O}$ is small.

Let $\mathcal R=\bigcap_{n\ge1}U_n$ be residual in $\mathcal Y$, 
with each $U_n$ open and dense. 
Then $\Phi^{-1}(U_n)\cap\mathcal O$ is open in $\mathcal O$ 
and dense by the local surjectivity of~$\Phi$. 
Hence
\[
\widetilde{\mathcal R}\cap\mathcal O
=
\bigcap_{n\ge1}
\bigl(\Phi^{-1}(U_n)\cap\mathcal O\bigr)
\]
is residual in $\mathcal O$.
\end{proof}

\subsection{The proof of Lemma~\ref{lem:trans}}

Now we are ready to prove Lemma~\ref{lem:trans}.
It suffices to prove the lemma for $C^r$-families with $r \ge 4$, since any less regular smooth family can be approximated by a $C^4$-family. Assume that $r \ge 4$ and  let~$F$ be an arbitrary $C^r$-smooth two-parameter family. Consider the set of $q$-periodic points of $F$:
\[
  P = \{(s, \theta, x)\in I_s\times I_\theta \times S^1
  \colon F^q_{(s,\theta)}(x) = x\},
\]
where $q$ is the denominator of the rotation number $q_n$ of interest.
The argument splits into two parts that have similar structure.
First we prove that for a generic $F$ for any $\beta$ the $q$-periodic points of the map $F_{\beta}$ are isolated, and hence there are finitely many of them. Each point $(\beta, x) \in P$ is a center of a small closed rectangle $U(\beta)\times I(x)$ that has disjoint images under the first $q-1$ iterates of the map
\[\operatorname{Id}\times F\colon (\beta, x) \mapsto (\beta, F_\beta(x)).\]
We cover $P$ by the corresponding open rectangles and take a finite subcover $(R_i)_{i=1}^N$. Treating each $R_i$ as $U \times V_0$ from Proposition~\ref{prop:gen_iter}, we recover the appropriate $V_j, j=1, \dots q$; we denote these by $V_{i,j}$ since they differ for different~$i$. The role of the set $\mathcal{R}$ is played by the set $\mathcal{R}_i$ of $G \in C^r(R_i, \operatorname{int}V_{i,q})$ with jet-extensions that do not intersect $\Sigma_{\mathrm{deg, 1}}$. Proposition~\ref{prop:gen_iter} yields that there is a locally generic set $\tilde{\mathcal{R}}_i$ in a neighborhood of $F$ such that for $\widehat{F}$ in this set $\widehat{F}^q|_{R_i}$ has no fixed points where $\partial_{xxx}\widehat{F}^q_\beta$ vanishes. Observe that for a $C^3$-smooth circle diffeomorphism $\partial_{xxx}$ must vanish at a non-isolated fixed point, by iterative application of Rolle's theorem. Hence, there are no non-isolated $q$-periodic points of $\widehat{F}$ in~$R_i$. Hence, any $\widehat{F} \in \cap_i \tilde{\mathcal{R}}_i$ does not have non-isolated $q$-periodic points in $A := \cup_i R_i$, but for $\widehat{F}$ close to $F$ there are no $q$-periodic points outside of~$A$. The global claim is obtained by applying this argument to $F$ from a countable dense set of families.

Now we will essentially repeat the argument above, but for $F$ with isolated $q$-periodic points and with $(R_i)_i$ replaced by a more refined finite cover for~$P$ that will be more convenient when considering the multijets. We want the set $P$ to be covered by the interiors of the sets $P_i,\, i=1, \dots, N$ with the following properties.

\begin{enumerate}
  \item Each $P_i$ is a product of a closed ball $B_i$ in the parameter space and a finite union $V_i \subset S^1$ of segments in the phase space.
  \item For every $i$, $V_i$ splits into a union $\sqcup_{j=0}^{q-1} V_{i,j}$ such that each $V_{i,j}$ is a finite disjoint union of segments and for every $\beta \in B_i$ we have
  \[F_\beta(V_{i,j}) \subset \operatorname{int} V_{i,j+1},\; j=0, \dots, q-2.\]
  \item There exists an $\varepsilon>0$ such that for any $\varepsilon$-perturbation of $F$ the set of $q$-periodic points is contained in the interior\footnote{The interior is taken w.r.t. the topology of $I_s\times I_\theta \times S^1$. Note that since the complement of this interior is compact sufficiently small perturbations cannot create $q$-periodic points outside this interior.} of $\bigcup_{i=1}^N P_i$. 
\end{enumerate}
To obtain such a cover, we consider an arbitrary $(\beta, x) \in P$ and cover the isolated $q$-periodic points of $F_\beta$ by small segments in $S^1$; denote the union of these segments by~$V_\beta$.
We multiply~$V_\beta$ by a small closed ball $U(\beta)$ around~$\beta$. Note that if $U(\beta)$ is sufficiently small then there are no points of $P \cap (U(\beta)\times S^1)$ not covered by $U(\beta) \times  V_\beta$. Moreover, if the segments in $V(\beta)$ are chosen properly, $V_\beta$ splits into a union of $V_{\beta, j}, \, j = 0,\dots, q-1,$ as in Proposition~\ref{prop:gen_iter}.
We cover $P$ by the interiors of $U(\beta)\times V_\beta$ and take a finite subcover. That is, $P_i$ are $U(\beta_i)\times V_{\beta_i}$ for some~$\beta_i$.

We take an $\varepsilon/2$-approximation of $F$ such that the restrictions of its $q$-th iterate to each $V_{i,0}$ have the transversality properties of Proposition~\ref{prop:multijet-transversality} inherited by Proposition~\ref{prop:gen_iter}.
More formally, we apply Proposition~\ref{prop:multijet-transversality} with $U \times A = B_i \times V_{i,0}$ and appropriately chosen $E$ to obtain a residual set $\mathcal{R}_i$ of maps that have multijet extensions transverse to all relevant submanifolds in the multijet space, and then apply to this set Proposition~\ref{prop:gen_iter}.
This yields a locally residual set $\tilde{\mathcal R}_i$, in a neighborhood of~$F$, of maps that have $q$-iterates whose restrictions to $B_i \times V_{i,0}$ are in~$\mathcal{R}_i$.

Consider another open and dense set $\tilde{\mathcal R}_0$ of two-parameter families that have structurally stable circle diffeomorphisms at the corners of $I_s \times I_\theta$.
Choose $\widehat{F} \in \tilde{\mathcal R} = \cap_{i=0}^N\tilde{\mathcal R}_i$ such that $\operatorname{dist}(\widehat{F}, F) <\varepsilon/2$.
For this family~$\widehat{F}$, all $q$-periodic points are contained in $\bigcup_{i=1}^N P_i$ and the restrictions $\widehat{F}^q|_{P_i}$ have jet-extensions transverse to submanifolds of Proposition~\ref{prop:multijet-transversality}. Recall that $\Sigma_{\mathrm{par}}$ has codimension~2. Transversality to $\Sigma_{\mathrm{par}}$ implies that the set $W \subset P$ of non-hyperbolic $q$-periodic points, being the $j^1(\widehat{F}^q)$-preimage of $\Sigma_{\mathrm{par}}$, is a 1-dimensional submanifold in $I_s\times I_\theta \times S^1$ with boundary contained in~$\partial(I_s\times I_\theta \times S^1)$, excluding $C\times S^1$, where $C$ is the set of corners. The projection of $W$ to $I_s\times I_\theta$ may have critical values and self-intersections, i.e., points with multiple preimages. The points with more than 2 preimages are ruled out by the transversality of ${}_3j^1(\widehat{F}^q)$ to $\Sigma^{(3)}_{\mathrm{par}}$; coexistence of a degenerate parabolic point and any other parabolic point is ruled out by the transversality of ${}_2j^2(\widehat{F}^q)$ to $\Sigma^{(2)}_{\mathrm{deg,par}}$. Hence, self-intersections correspond to the coexistence of two quadratic parabolic orbits. These self-intersections are transverse since ${}_2j^1(\widehat{F}^q)$ does not intersect $\Sigma^{(2)}_{\mathrm{par,tan}}$. Therefore, the intersections happen at isolated points. The critical values for the projection are exactly the cusp bifurcation values; other possibilities are ruled out by the image of $j^3(\widehat{F}^q)$ having no intersections with $\Sigma_{\mathrm{deg,1}}$ and $\Sigma_{\mathrm{deg,2}}$. The projection $\pi_B(W)$ of $W$ to $B$ has finitely many points with horizontal tangents since $j^1(\widehat{F}^q)$ is transverse to $\Sigma_{\mathrm{par,hor}}$.

Consider one of the edges $J_1$ of the boundary $\partial B$, say, $J_1 = \{0\}\times I_\theta$. After intersecting $\tilde{\mathcal{R}}$ with another residual subset of families, we may assume that the restriction of $\widehat{F}$ to $J_1$ has multijet extensions transverse to the submanifolds $\Sigma_{\mathrm{deg}}, \Sigma^{(2)}_{\mathrm{par}}, \Sigma_{\mathrm{par, hor}}$. This shows that for $\widehat{F}$ there are no cusp bifurcations, coexisting parabolic points, or tangencies with $\pi_B(W)$ at any point of~$J_1$. For a vertical segment $J_2$ in $\partial B$, the argument is analogous, but one has to consider the submanifolds $\Sigma_{\mathrm{par,vert}}$ defined analogously to $\Sigma_{\mathrm{par, hor}}$, but with respect to points where $\pi_B(W)$ has vertical tangent lines.

We have proven the claims of Lemma~\ref{lem:trans} for $\pi_B(W)$ rather than the subset $S_{\mathrm{par, q_n}}$ of $B$ that corresponds to diffeomorphisms with rotation number $q_n$ and a parabolic periodic orbit. But since the rotation number takes values in a discrete subset of $\mathbb{Q}$ for parameters in $\pi_B(W)$ and depends continuously on the map, it is constant on each connected component of~$\pi_B(W)$. But $S_{\mathrm{par, q_n}}$ is the union of connected components of $\pi_B(W)$, where the rotation number equals $q_n$, hence $S_{\mathrm{par, q_n}}$ has the same properties. The proof of Lemma~\ref{lem:trans} is complete.

\begin{rem}\label{rem:isolated}
Now it is easy to see that a generic $C^r$-smooth $l$-parameter family of circle diffeomorphisms has isolated periodic points for each parameter value~$\beta$ provided that $r > l+1$. Indeed, consider $q$-periodic points. In the space of $(l+1)$-jets, we can consider the analogues $\Sigma_{l+1}$ of $\Sigma_{\mathrm{deg,1}}$ defined as
\[
   \Sigma_{l+1} = \{\, j^{l+1}_{(\beta, x)}F \colon F_\beta(x)=x, \; \partial_xF_\beta(x)=1, \; (\partial_x)^jF_\beta(x)=0, \, j=2,\dots,l+1\}.
\]
Arguing as in the first part of the proof of Lemma~\ref{lem:trans}, we obtain that the $q$-th iterate of a generic $l$-parameter family $F$ has an $(l+1)$-jet extension with domain of dimension $l+1$ and with image that does not intersect~$\Sigma_{l+1}$ that has codimension $l+2$. This implies that $q$-periodic points are isolated. It suffices to intersect the corresponding residual subsets of families obtained for different~$q$.
\end{rem}

\begin{rem}\label{rem:nondenfam}
The argument in the proof of Lemma~\ref{lem:trans} can be applied to one-parameter families with minor modifications to prove that \emph{non-degenerate} families form an open and dense set in the space of $C^r$ one-parameter families. In this case, we only consider $\Sigma_{\mathrm{par}}, \Sigma_{\mathrm{deg}}$ and $ \Sigma_{\mathrm{par, hor}}$. Since the source manifold $I_\theta \times S^1$ is compact, by Corollary~5.7 of~\cite{F} the sets $\mathcal{R}_i$ are not just residual but open and dense, and the sets $\tilde{\mathcal R}_i$ inherit this property. We intersect finitely many of these and obtain a residual set $\tilde{\mathcal R}$ of one-parameter families that exhibit only quadratic parabolic bifurcations at isolated values of the parameter. Note that we used the transversality to $\Sigma_{\mathrm{par, hor}}$ to exclude the case when the parabolic point is present, but the bifurcation does not unfold in the family. Finally, if a family has a diffeomorphism with more than one parabolic orbit, a small perturbation yields a family where this is not the case. By transversality to $\Sigma_{\mathrm{par}}$, the preimage $W$, which is a finite set now, continuously depends on the family, so a small perturbation of the new family will not have coexisting parabolic points. 
\end{rem}

\end{document}